\documentclass[11pt]{amsart}
\usepackage{amsmath}
 \usepackage{amscd}
 \usepackage{amssymb}

\usepackage{amsmath}
\usepackage{hyperref}
\usepackage[applemac]{inputenc}
\usepackage{tikz-cd}
\usepackage{ulem}
\usepackage[applemac]{inputenc}
\usepackage[all]{xy}
\usepackage{color}

\usepackage{xparse}

\usepackage{xy}
\usepackage{amscd}
\newfont{\sheaf}{eusm10 scaled\magstep1}

\newtheorem{definition}{Definition}[section]

\newtheorem{proposition}{Proposition}[section]

\newtheorem{theorem}[proposition]{Theorem}

\newtheorem{remark}{Remark}[section]

\DeclareMathOperator{\codim}{codim}

\DeclareMathOperator{\Supp}{Supp}
\DeclareMathOperator{\Sing}{Sing}

\DeclareMathOperator{\Aut}{Aut}

\DeclareMathOperator{\Pic}{Pic}
\DeclareMathOperator{\Sim}{Sym}

\DeclareMathOperator{\rank}{rank}

\title[NOTES]{\small From Enriques surface to Artin-Mumford counterexample }
\author[  A. Verra]{Alessandro Verra}
 \address{Dipartimento di Matematica \\ %
Universit\'a degli Studi di Roma TRE \\ %
Largo San Leonardo Murialdo \\ %
00146 Roma \\ %
Italy} \thanks{This work was partially supported by INdAM-GNSAGA}
 \email{sandro.verra@gmail.com}
 
\begin{document} \maketitle 
\begin{abstract} {After an Introduction to the themes of Enriques surfaces and Rationality questions, the Artin-Mumford counterexample to L\"uroth problem is revisited.   A construction of it is given, which  is related in an explicit way to the geometry of    Enriques surfaces, more precisely to the special family of Reye congruences and their classical geometry.}   \end{abstract}
\section{Perspectives on Enriques surface and Rationality } 
[\footnote{ \tiny We work overt the complex field.}]
This section   serves as a non technical introduction to this paper and to a wider theme of investigation, in algebraic geometry and its history, we can well summarize by 
 the key words \it Enriques surface \rm and \it rationality of algebraic varieties. \rm 
 These words certainly represent central and related issues, typical of the Italian contribution to algebraic geometry: since the golden age of the classification of complex algebraic surfaces,
 by Castelnuovo and Enriques, to the present times. Their presence is therefore natural in the Proceedings of an INdAM workshop bearing the title \it The Italian contribution to
 Algebraic Geometry between tradition and future.\rm \par
 In particular,  the names of Castelnuovo and Enriques are definitely related to  \it Castelnuovo Criterion of Rationality \rm and to  the discovery of \it Enriques surfaces. \rm 
 As is well known these achievements, in the theory of complex algebraic surfaces and their birational classification,  are strictly related and represent one of the culminating points, in the history of algebraic geometry and of the Italian contribution, during  some fortunate years at the juncture of 19th and 20th centuries. \par Without entering in technical issues, the discovery of Enriques surfaces was motivated, independently from the work of birational classification, by the search of counterexamples to a quite natural conjecture, concerning the rationality of an algebraic surface $S$. More precisely, it was somehow natural to expect that  $S$ is rational if and only if its geometric genus $p_g(S)$ and its irregularity $q(S)$ are zero. \par
  This led Enriques, in the long days of a summer vacation of 1894,  to consider the family of sextic surfaces, in the complex projective space $\mathbb P^3$, having multiplicity two along the edges of a tetrahedron. He proved that a general such a sextic is non rational, though its geometric genus and irregularity are zero. This is a celebrated episode, see the correspondence Castelnuovo-Enriques \cite{CE} p. x and 111.  \par In this way an entire class of surfaces was discovered and nowadays these bear the name of Enriques surfaces. Actually any  Enriques surface turns out to be birational to a sextic as above. Castelnuovo's Criterion was proven in the same period: let $P_2(S)$ be the bigenus  of $S$,  then  $S$ is rational iff
  $$ P_2(S) = q(S) = 0.$$
\par A general motivation for this paper is to stress that, along their more than centennial history, Enriques surfaces certainly did not stay confined in the classification of algebraic surfaces as a kind of special or exhotic plant.  Instead their original interaction with rationality problems in dimension two impressively started to extend to higher dimension, implicitly or explicitly, touching in particular the famous L\"uroth's problem. \par
 L\"uroth's problem is the question of deciding whether an algebraic variety $V$ admitting rational parametric equations $f: \mathbb C^N \to V$ is also rational, that is, it admits birational parametric equations
$g: \mathbb C^d \to V$, $d = \dim V$. \par
L\"uroth's theorem and Castelnuovo's Criterion imply this property for curves and surfaces, but the problem was staying open for years in dimension $\geq 3$. This until the crucial 1971-1972, when three 
counterexamples came out in dimension $3$: by Artin-Mumford, Clemens-Griffiths, Manin-Iskovskikh. This  famous triple episode can be considered as the beginning of
a 'Modern Era' in the domain of rationality problems, see \cite{B2}-1.3 and section 4.1. \par
Whatever it is, starting from Castelnuovo and Enriques and considering all the decades until the present days of Modern Era, the overlapping of perspectives, on the mentioned themes of Enriques surfaces and rationality problems, are often visible and always very interesting. \par As an example, let us mention the debate on L\"uroth's problem, originated in the Fifties from Roth's book \cite{R}, and the related Serre's theorem that unirational implies simply connected, \cite{S}. The threefold considered  is a unirational sextic in $\mathbb P^4$ whose hyperplane sections are sextic Enriques surfaces. The discussion was about, possibly, proving its irrationality via some features of irrationality of its hyperplane sections, see \cite{B3}, \cite{PV}. \par We cannot
add more views and perspectives in this Introduction, but mentioning the extraordinary wave of breakthrough results from the very last years. This is true in particular for stable rationality of some threefolds and cohomological decomposition of their diagonal, see \cite{V1, V2, V3}. \par
Coming to the contents of this paper, we revisit the counterexample of Artin-Mumford, putting in evidence the presence, sometimes behind the scene in the literature, of an Enriques surface $S$. $S$ is embedded in the Grassmannian $\mathbb G$ of lines of $\mathbb P^3$ and well known as a Reye congruence. \par
$S$ brings us in the middle of classical algebraic geometry: $S$ is obtained from a general $3$-dimensional linear system $W$, a web, of quadric surfaces. Moreover the threefold to be considered is the double covering of $W$, 
$$
f: \tilde W \to W,
$$
parametrizing the rulings of lines of the quadrics of $W$. It is easy to see that $\tilde W$ is unirational, see 4.7 and \cite{B3}.  The branch surface of $f$  is a Cayley quartic symmetroid $\tilde S_+$, parametrizing the singular quadrics of $W$. A natural desingularization $\tilde W'$ of $W$ is the blow up of $\tilde W$ at its ten singular points. Artin and Mumford prove that, for a smooth variety $V$, the torsion of $H^3(V, \mathbb Z)$ is a birational invariant. Moreover they prove 
$$
H^3(\tilde W', \mathbb Z) \cong \mathbb Z/2\mathbb Z,
$$
which implies the irrationality of $\tilde W'$, since $H^3(\mathbb P^3, \mathbb Z) = 0$. As is well known the Enriques surface $S$ has the same feature of irrationality:
$$ H^3(S, \mathbb Z) \cong H_1(S, \mathbb Z) \cong \mathbb Z/ 2\mathbb Z. $$
We describe, and use, the nice geometry offered by $S$ and by the Fano threefold $\tilde W$, to reconstruct explicitly, from the nonzero class of $H_1(S, \mathbb Z)$, the nonzero class of $H^3(\tilde W', \mathbb Z)$. In particular we profit of the special feature of the Fano surface of lines of $\tilde W$. Indeed this is split in two irreducible components birational to $S$, see 4.4 and \cite{F2}.   
\begin{remark} \rm For $t \geq 3$, let $W_t$ be a general linear system of dimension $\binom t2$ of quadrics of $\mathbb P^t$. Let $W^4_t \subset W_t$ be the locus of quadrics of rank $\leq 4$. Then  $\dim W^4_t = 2t - 3$ and we have a double covering $f: \tilde W^4_t \to W^4_t$, parametrizing the rulings of subspaces of maximal dimension of the quadrics of $W^4_t$. A special family of these linear systems  is associated to Reye congruences as in \cite{CV}. In this case $W_t$  defines an embedding of a Reye congruence  $S$ as a surface of class $(t, 3t-2)$ in the Grassmannian $\mathbb G_t$ of lines of $\mathbb P^t$. It seems that, like for $t = 3$, $S$ determines nonzero $2$-torsion in $H^3(\tilde W_t^{ '4}  , \mathbb Z)$, $\tilde W_t^{ '4}  $ being a desingularization of $\tilde W_t^{ 4}  $. This will be possibly reconsidered elsewhere.
 \end{remark} 
  { \it Aknowledgements}  \par We are grateful to   Bert van Geemen, John Ottem and Claire Voisin for some useful conversations  and advice on the subject of this paper.  Let us also thank the Scientific Committee of this very interesting conference for its realization.

\section{Webs of quadrics and Enriques surfaces}
 Now we introduce the family of those Enriques surfaces bearing the name of \it Reye congruences,\rm [\footnote{ \tiny See \cite{DK} 3.7 for a historical note.}] Let $\mathbb G$ be the Grassmannian of lines of
$\mathbb P^3$,  a general member $S$ of this family is a smooth Enriques surface embedded in $\mathbb G$ and such that its rational equivalence class
in the  Chow rindg $\rm{CH}^*(\mathbb G)$ is 
\begin{equation}
7 \sigma_{2,0} + 3\sigma_{1,1} \ \ \  [\footnote{\tiny$\sigma_{1,1}$, ($\sigma_{2,0}$),  is the class of the set of lines in a plane, (through a point).}].
\end{equation}
  Then $S$ has degree $10$ and sectional genus $6$ in the Pl\"ucker embedding of $\mathbb G$. Moreover $\mathcal O_S(1)$ is an example of Fano polarization of an Enriques surface, \cite{CDL} 3.5. Notice that, in the period space of Enriques surfaces,  the locus of Reye congruences is an irreducible divisor and coincides with the locus of periods of Enriques surfaces containing a smooth rational curve, \cite{N}. 
The construction of these surfaces was given by Reye, \cite{R}. It relies on projective methods and the beautiful geometry of webs of quadric surfaces. 
 We concentrate on these related topics, recovering a general picture.
 \par  \subsection{\it Webs of quadrics} \  \par Let $E$ be a $4$-dimensional vector space, we fix the notation 
 \begin{equation} \mathbb P^3 := \mathbb P(E) \end{equation}  and define the Pl\"ucker embedding of the Grassmannian of lines of $\mathbb P^3$ by
\begin{equation}
\mathbb G \subset \mathbb P^{5-},
\end{equation}
where $\mathbb P^{5-} := \mathbb P(\wedge^2 E)$. In particular $\mathbb G$ is a smooth quadric hypersurface. Then we consider $E \otimes E$ and its standard direct sum decomposition
\begin{equation}
E \otimes E = \wedge^2 E \oplus \Sim^2 E,
\end{equation}
via the eigenspaces of the involution exchanging the factors of $E \otimes E$.  
The induced involution on $\mathbb P^{15} := \mathbb P(E \otimes E)$ will be denoted by
\begin{equation}
\iota: \mathbb P^{15} \to \mathbb P^{15}.
\end{equation}
We also put $\mathbb P^{9+} := \mathbb P(\Sim^2 E)$ and observe that the set of fixed points of $\iota$ is $\mathbb P^{5-} \cup \mathbb P^{9+}$. Then we define the commutative diagram  \begin{equation} \label{DIAGRAM}
\begin{CD}
{ \mathbb G} @<<< {\mathbb P^3 \times \mathbb P^3} @>>> {\mathbb S} \\
@VVV @VVV @VVV \\
{\mathbb P^{5-}} @<{\lambda_{-}}<< {\mathbb P^{15}} @>{\lambda_+}>>{\mathbb P^{9+}}. \\
\end{CD}
\end{equation}  
Here $\mathbb S$ is the set of points defined by  symmetric tensor $a \otimes b + b \otimes a$ having rank $\leq 2$. Notice that $\mathbb S$ is biregular to the quotient $\mathbb P^3 \times \mathbb P^3 \slash \langle \iota \rangle$. The vertical arrows are the natural inclusions. Moreover $\lambda_+$ and $\lambda_-$ 
are the natural linear projections onto the spaces $\mathbb P^{9+}$ and $\mathbb P^{5-}$.   
The restrictions of  $\lambda_-$ and $\lambda_+$ to $\mathbb P^3 \times \mathbb P^3$ admit an elementary description. Let  $(x,y) \in \mathbb P^3 \times \mathbb P^3$ and let $\ell \in \mathbb G$, where $\ell$ is the line through $x,y$ if $x \neq y$, then we have
\begin{equation}
\lambda_-(x,y) = \ell \in \mathbb G \ \ \text{and}  \ \ \lambda_+(x,y) = x+y \in \mathbb S.
\end{equation}
\begin{definition} $\mathbb T^r$ is the set of points in $\mathbb P^{15}$ defined by tensors $t \in E \otimes E$ having rank $\leq r$.  In particular  $\mathbb T^2$ is $\mathbb P^3 \times \mathbb P^3$ and  $\mathbb T^1$ is its diagonal.   \end{definition}
From the loci $\mathbb T_r$ we have of course the rank stratification
\begin{equation}
\mathbb T^1 \ \subset  \ \mathbb T^2 \ \subset \  \mathbb T^3 \ \subset \ \mathbb P^{15}.
\end{equation}
  Now we pass to the dual space $\mathbb P^{15*}$ of bilinear forms, to be denoted by
 \begin{equation}
\mathbb  B := \mathbb P(E^* \otimes E^*).
\end{equation}
Let $b \in \mathbb B$ then $b^{\perp} \subset \mathbb P^{15}$ will denote its orthogonal hyperplane. Moreover let $W \subset \mathbb B$ then the orthogonal subspace of $W$ is by definition
\begin{equation}
 W^{\perp} := \bigcap_{b \ \in W} b^{\perp}.
 \end{equation}
 Let $W$ be a subspace of dimension $c$ then $\codim {W^{\perp}} = c+1$ and we have
 \begin{equation}
  \vert \mathcal I_{W^{\perp}}(1) \vert = W,
 \end{equation}
 where we reserve the notation $\mathcal I_{W^{\perp}}$ to the ideal sheaf of $W^{\perp}$ in $\mathbb P^{15}$. Notice that $\mathbb B$ contains the subspace
 $\mathbb Q := \mathbb P(\Sim^2 E^*)$ and that this is just the space of quadrics of $\mathbb P^3$. We denote its rank stratification by
\begin{equation} \label{QUADRICS}
\mathbb Q^1 \ \subset  \ \mathbb Q^2 \ \subset \  \mathbb Q^3 \ \subset \ \mathbb Q.
\end{equation}
 Finally we come to web of quadrics of $\mathbb P^3$. We use the traditional word \it web \rm for a $3$-dimensional linear system of divisors of a variety. For short we will use the word \it web of quadrics \rm for a web of quadric surfaces of $\mathbb P^3$.  Let 
 \begin{equation}
 W \subset \mathbb Q \subset \mathbb B
 \end{equation}
  be a general web of quadrics, then $W$ naturally defines an Enriques surface which is a Reye congruence, \cite{DK} 7. We construct it as follows. We have
 \begin{equation} W = \mathbb P(V), \end{equation} 
where $V \subset \Sim^2 E^*$ is a general $4$-dimensional space. Notice that $\iota^*$ is the identity on $V = H^0(\mathcal I_{W^{\perp}}(1))$, therefore the
$5$-dimensional subspace  
\begin{equation} W^{\perp} \subset \mathbb P^{15} \end{equation} satisfies $\iota(W^{\perp}) = W^{\perp}$.  Now observe that the set of fixed points of the involution $\iota \vert W^{\perp}$ is the disjoint union of
the space $\mathbb P^{5-} \subset W^{\perp}$ and of
\begin{equation}
    W^{\perp +} := W^{\perp} \cap \mathbb P^{9+}.
\end{equation}
Therefore $\mathbb P^{5-} \cup W^{\perp +}$ generates $W^{\perp}$ and $\codim W^{\perp +} = 4$ in $\mathbb P^{9+}$, so that
 \begin{equation}
 \lambda_+^*(W^{\perp +}) = W^{\perp}.
 \end{equation}
 This implies that the family of spaces $W^{\perp +}$ coincides with the Grassmannian of codimension $4$ spaces in $\mathbb P^{9+}$. Notice also that the family
 of spaces $W^{\perp}$ is the family of codimension $4$ spaces of $\mathbb P^{15}$ containing $\mathbb P^{5-}$.  Now  consider  
 \begin{equation}
 \lambda_+ \vert \mathbb T_2: \mathbb T_2 \to \mathbb S 
 \end{equation}
as in (\ref{DIAGRAM}). This is the quotient map of $\iota \vert \mathbb T_2$ and a finite double cover branched on $\Sing \ \mathbb S$, that is,  on the image of the diagonal $\mathbb T_1$ of $\mathbb T_2 = \mathbb P^3 \times \mathbb P^3$.   
\par \subsection{ \it Reye congruences of lines} \ \par
\begin{definition} Given a web $W$ let us fix the notation
\begin{equation}
S_+ := \mathbb S \cdot W^{\perp +} \ \ , \ \ \tilde S := \mathbb T_2 \cdot W^{\perp}.
\end{equation}
\end{definition}
We assume that $W$ is general, so that $W^{\perp +}$ is disjoint from $\Sing \mathbb S$ and transversal to the map $\lambda_+$. By Bertini theorem, $\tilde S_+$ and $S$ are smooth, irreducible surfaces and linear sections of $\mathbb T_2$ and $\mathbb S$. It is also clear that
 \begin{equation}
\iota \vert \tilde S: \tilde S \to \tilde S.
\end{equation}
is a fixed point free involution. Indeed $\iota^*$ is the identity on $V$ and $\tilde S$ is disjoint from the set $\mathbb T_1$ of fixed points of $\iota$. Hence $\lambda_+: \mathbb T_2 \to \mathbb S$ restricts to an \'etale double covering
 $
{\lambda_+}\vert \tilde S: \tilde S \to S_+.
$
To complete the picture, we consider the rational map $\lambda_-: \mathbb P^{15} \to \mathbb P^{5-}$ and its restriction $\lambda_- \vert \tilde S$, we fix the notation $S := \lambda_-(\tilde S)$. Of course we have
\begin{equation}
S  \subset \mathbb G \subset \mathbb P^{5-}.
\end{equation}
   \begin{theorem}  The surfaces $S$ and $S_+$ are embeddings, of degree $10$ and sectional genus $6$,  of the same Enriques surface $\tilde S / \langle \iota \vert \tilde S \rangle$.  Moreover one has 
\begin{equation}
\omega_{S_+}(1) \cong \mathcal O_S(1).
\end{equation}
\end{theorem}
\begin{proof} $\tilde S$ is a K3 surface, endowed with the fixed point free involution $\iota \vert \tilde S$ and embedded in $\mathbb T_2 = \mathbb P^3 \times \mathbb P^3$ as a complete intersection  of four elements of $\vert \mathcal O_{\mathbb P^3 \times \mathbb P^3}(1,1) \vert$. This follows by its construction and adjunction formula. Now $S$ and $S_+$ are embeddings of the same smooth surface $\tilde S / \langle \iota \vert \tilde S \rangle$. Since $\iota \vert \tilde S$ is fixed point free and $\tilde S$ is a smooth K3 surface, then $\tilde S / \langle \iota \rangle$ is a smooth Enriques surface. On the other hand the K3 surface  $\tilde S$ is embedded in $W^{\perp}$ by   $\mathcal O_{\tilde S}(1) := \mathcal O_{\mathbb P^3 \times \mathbb P^3}(1,1) \otimes \mathcal O_{\tilde S}$, a polarization  of genus $11$ and degree $20$. Moreover the involution $\iota \vert \tilde S$ acts on $H^0(\mathcal O_{\tilde S}(1))$ and its quotient map $\lambda_+ \vert \tilde S: \tilde S \to S_+$ is the double cover defined by the canonical sheaf $\omega_{S_+}$ and $\lambda_+^* \mathcal O_{S_+}(1)$
$\cong$ $\mathcal O_{\tilde S}(1)$. Then it follows that we have the decomposition  
\begin{equation}
H^0(\mathcal O_{\tilde S}(1)) = \lambda_+^* H^0(\mathcal O_{S_+}(1)) \oplus \lambda_+^* H^0(\omega_{S_+}(1)).
\end{equation}
The summands respectively are the $6$-dimensional $+1$ and $-1$ eigenspaces of $(\iota \vert \tilde S)^*$. This implies the last part of the statement, we omit some details.
 \end{proof} 
 Clearly $\lambda_-\vert \tilde S: \tilde S \to \mathbb P^{5-}$ factors through $\iota \vert \tilde S$ and defines the embedding
 \begin{equation}
 S \subset \mathbb G \subset \mathbb P^{5-}.
 \end{equation}
  Both $S$ and $S_+$ are examples of Fano models of an Enriques surface.  $S$ is known as a \it
 Reye congruences of lines.  \rm A general Fano model is projectively normal,  hence it is not contained in a quadric, cfr. \cite{CDL}, 3.5.  Instead  $S$ is contained in the smooth quadric $\mathbb G$. This is the special feature of this family. Each $S$ is endowed with a special rank two vector bundle, namely the restriction of the universal bundle of $\mathbb G$. We will use it to introduce the classical construction of $S$, see \cite{Da}, \cite{R} and \cite{DK}-7.
 \subsection{\it Further notation} \ \par The universal bundle over $\mathbb G$ is $ p: U \to \mathbb G$,  then $U_{\ell} \subset E$ and we define 
  \begin{equation}  \mathbb U_{\ell} := \mathbb P(U_{\ell}) \subset \mathbb P^3, \end{equation} for each $\ell \in \mathbb G$. This is the line in $\mathbb P^3$ defining the point $\ell \in \mathbb G$. After some tradition, we say that a $\ell$ is \it a ray \rm of $\mathbb G$. We will also use the $\mathbb P^3$-bundle 
  \begin{equation} \pi: \mathbb P(U \otimes U) \to \mathbb G, \end{equation} whose fibre at $\ell$
is $\mathbb P^3_{\ell} := \mathbb P(U_{\ell} \otimes U_{\ell})$. Let  $ \tau: \mathbb P(U \otimes U) \to \mathbb P^{15}$
be its tautological map, then $\tau_{\ell}: \mathbb P^3_{\ell} \to \mathbb P^{15}$ is the linear embedding defined by the inclusion of $U_{\ell} \otimes U_{\ell} \subset E \otimes E$. Moreover $\tau$ is a morphism  birational onto its image. We identify $\mathbb P^3_{\ell}$ to its image by $\tau_{\ell}$ so that  $\mathbb P^3_{\ell} \subset \mathbb P^{15}$.
 Since  we have \begin{equation}  U_{\ell} \otimes U_{\ell} =  \wedge^2 U_{\ell} \oplus \Sim^2 U, \end{equation} then $\mathbb P^3_{\ell}$  is $\iota$-invariant. Its point $\ell = \mathbb P(\wedge^2 U_{\ell})$ and its plane $Q_{\ell} := \mathbb P(\Sim^2 U_{\ell})$ are the set of fixed points of $\iota \vert \mathbb P^3_{\ell}$. We observe that $\mathbb T_2 \cap \mathbb P^3_{\ell}$ is the quadric  $\mathbb U_{\ell} \times \mathbb U_{\ell}$ and $ \mathbb T_1 \cap \mathbb P^3_{\ell}$ its diagonal.  Finally, the next diagram will be useful:
 \begin{equation}
\begin{CD}
{\mathbb P^{15}} @<{\tau}<< {\mathbb P(U \otimes U)} @>{\pi}>>{\mathbb G}. \\
@AAA @AAA @AAA \\
{ \mathbb T_2} @<{\delta}<< {\mathbb D} @>{\pi}>> {\mathbb G} \\
 \end{CD}
\end{equation}   
Its vertical arrows are the natural inclusions and $\mathbb D$ is the projectivized set of points defined in $\mathbb P(U \otimes U)$ by the decomposable tensors in $U \otimes U$.   
\section{ Reye congruences and Symmetroids}
\subsection{\it The classical construction} \ \par It is now useful to revisit in modern terms Reye's classical construction of the surface $S$, cfr. \cite{R, Da, Co, DK}. Let $W$ be a general web of quadrics, defining as above a smooth $K3$ surface $\tilde S$ and the embeddings $S \subset \mathbb G$ and $S_+ = \mathbb S \cdot W$ of $S$. We assume  $W = \mathbb P(V)$, then $V$ is a space of quadratic forms on $E$ and we have the natural inclusions
  \begin{equation}
  V \subset H^0(\mathcal O_{\mathbb P^3}(2)) \subset H^0(\Sim^2 U^*).
  \end{equation} 
  The evaluation of global sections of $\Sim^2 U^*$ defines a morphism  
 \begin{equation} \label{EVALUATION}
 e: \mathcal O_{\mathbb G} \otimes V \to \Sim^2{ U}^*,
 \end{equation}
 of vector bundles of ranks $4$ and $3$.  Counting dimensions, the degeneracy scheme of $e$ is a surface, provided it is proper and non empty.
\begin{theorem} The degeneracy scheme of $e$ is the Enriques surface $S$ and its rational equivalence class is $7 \sigma_{1,1} + 3\sigma_{2,0}$ in ${\rm CH}^*(\mathbb G)$. 
 \end{theorem}
 \begin{proof} Assume the degeneracy scheme $S_e$ of $e$ is proper.  Then, computing its class in the Chow ring of $\mathbb G$, we obtain that $[S_e]$ $=$ $7 \sigma_{1,1} + 3\sigma_{2,0}$ in ${\rm CH}^*(\mathbb G)$. Hence $\deg S_e = \deg S = 10$ and the equality $\Supp S_e = S$ implies the statement. To show the equality consider $\ell \in \mathbb G$ and the fibrewise map
\begin{equation}
e_{\ell}: V \to H^0(\mathcal O_{\mathbb U_{\ell}}(2)).
\end{equation}
This is the restriction map to the line $\mathbb U_{\ell} \subset \mathbb P^3$, defined as above.   Equivalently, keeping our previous identifications, we can assume that the curve
\begin{equation}
\Delta_{\ell} \subset \mathbb U_{\ell} \times \mathbb U_{\ell} \subset \mathbb P^3_{\ell} \subset \mathbb P^{15}
\end{equation}
is the diagonal embedding of $\mathbb U_{\ell}$, and that $V \subset H^0(\mathcal O_{\mathbb P^{15}}(1))$. Then $e_{\ell}$ is the restriction map $V \to H^0(\mathcal O_{\Delta_{\ell}}(1))$. Moreover we know that $V$ has a basis $a_1, a_2, a_3, s$ so that $\iota^* a_j = -a_j$ and $\iota^*s = s$, where $s$ is zero on $\Delta_{\ell}$. Hence $e_{\ell}$ degenerates
iff its rank is two. Equivalently $V \to H^0(\mathcal O_{\mathbb P^3_{\ell}}(1))$ defines a pencil of planes in $\mathbb P^3_{\ell}$ and its base line intersects $\mathbb U_{\ell} \times \mathbb U_{\ell}$
 in two distinct points $x, y \in \tilde S$ such that $y = \iota(x)$.  This implies $\Supp S_e = \lambda_-(\tilde S) = S$.
 \end{proof}
 The proof reveals the main geometric feature of $S$, we have:
\begin{equation}
S = \lbrace \ell \in \mathbb G \ \vert \ \dim (V \cap H^0(\mathcal I_{\ell}(2)) = 2 \rbrace.
\end{equation}
In other words $\ell \in S$ iff two quadrics of $W$ contain $\ell$, that is, $\ell$ is in the base locus of a pencil of quadrics of $W$. We can conclude as follows.
\begin{definition} \label{REYEDEGENERACY} The Reye congruence of $W$ is the degeneracy scheme $S$ of the previous morphism $e$, provided $S$ is proper. \end{definition}
\begin{theorem} \label{REYEGEOM} The Reye congruence of $W$ is the family of lines of $\mathbb P^3$ which are in the base locus of a pencil of quadrics contained in $W$.
\end{theorem}
\subsection{ \it Order and class of $S$} \ \par
Following the classical language the order of a surface $Y \subset \mathbb G$ is the number $a$ of rays of $Y$ passing through a general point and the class is the number
$b$ of rays of $Y$ in a general plane. Then one has $[Y] = a\sigma_{1,1} + b\sigma_{2,0}$ in $CH^*(\mathbb G)$.  Of course these notions naturally extend for surfaces in any Grassmannian of lines.  Let us motivate geometrically the equality $$ [S] = 7\sigma_{1,1} + 3\sigma_{2,0} \in {\rm CH}^*(\mathbb G). $$ \par The coefficient $7$ means that 
exactly seven rays of $S$ contain a general point $o \in \mathbb P^3$. Consider the net $W_o \subset W$ of quadrics through  $o$, then its  base locus  is a set of eight distinct points $\lbrace o \ o_1 \dots o_7\rbrace$. It is easy to see that the  lines $ \overline {oo}_1 \dots \overline {oo}_7$  are the seven rays of the family passing through $o$. \par
The coefficient $3$ means that exactly three rays of $S$ are contained in a general plane $P$. The web $W$ restricts on $P$ to a web $W_P$ of conics. How many lines
 in $P$ are fixed component of a pencil of conics of $W_P$? The answer is classical: $W_P$ is generated by four \it double \rm lines, supported on lines in general position. 
 These define a complete quadrilateral. One can check that its three diagonals are precisely the lines with the required property.
 \subsection{ \it The quartic symmetroid} \ \par
Now we concentrate further on $W$, considering the quartic surface
\begin{equation}
\tilde S_{+} = \mathbb Q^3 \cdot W
\end{equation}
parametrizing the singular quadrics of a general $W$. Clearly such a surface is defined, in the projective space $W$, by the determinant  of a 
symmetric $4 \times 4$ matrix of linear forms. For this reason it bears the following name.
\begin{definition} A quartic symmetroid is a surface $\tilde S_{+}$, constructed as above from a general web of quadrics $W$.
\end{definition}
This classical surface is well known.  Let $W \subset \mathbb Q$ be transversal to the quartic hypersurface $\mathbb Q^3 \subset \mathbb Q$ 
and to its singular locus $\mathbb Q^2$. Then, counting dimensions
and degrees, we have $\mathbb Q^1 \cap W = \emptyset$ and, moreover,  the set
\begin{equation} \Sing \tilde S_{+} = \mathbb Q^2 \cap \tilde {S}_+ \end{equation} consists of ten ordinary modes. Actually $\tilde S_+$ is a birational projective model of $\tilde S$, as we are going to see. To this purpose let us  fix on $\mathbb P^3 \times \mathbb P^3$ coordinates 
\begin{equation}
(x,y) := (x_1:x_2:x_3:x_4) \times (y_1:y_2:y_3:y_4),
\end{equation}
so that $\tilde S$ is defined by the four symmetric bilinear equations
\begin{equation} 
\sum_{1 \leq i,j \leq 4} q^{[k]}_{ij}x_iy_j  = 0,  \ \ \ k = 1 \dots 4.
\end{equation}
The set of quadratic forms $q^{[k]} := \sum q^{[k]}_{ij} x_ix_j$ is a basis for the vector space $V$ such that $W = \mathbb P(V)$. From now on we set $z := (z_1:z_2:z_3:z_4)$ and
\begin{equation}
q_z := z_1 q^{[1]} + z_2 q^{[2]} + z_3 q^{[3]} + z_4 q^{[4]},
\end{equation}
denoting by $Q_z$ the quadric in $\mathbb P^3$ defined by $q_z$. Then we compute that
\begin{equation}
\frac {\partial q_z}{\partial x_j} = z_1 q^{[1]}_j + z_2 q^{[2]}_j + z_3 q^{[3]}_j + z_4 q^{[3]}_j, \ \ \ j = 1 \dots 4,
\end{equation}
where we put
$
q^{[k]}_j := \sum_{1 \leq i \leq 4} q^{[k]}_{ij}x_i.
$
Clearly the four bilinear equations
\begin{equation}
\frac {\partial q_z}{\partial x_j} = 0, \ \ \  j = 1 \dots 4,
\end{equation}
define in the product $\mathbb P^3 \times W$  the incidence correspondence
\begin{equation}
\Xi :=  \lbrace (x, z) \in \mathbb P^3 \times W \ \vert \ x \in \Sing Q_z \rbrace.
\end{equation}
This is  the universal singular locus over the family of quadrics $W$. Let 
\begin{equation}
p_x: \Xi \to \mathbb P^3 \ \text{and} \ p_z: \Xi \to W, 
\end{equation}
be the projections of $\Xi$ in the factors of $\mathbb P^3 \times W$. Obviously $p_z(\Xi)$ is the quartic symmetroid $\tilde S_+$, defined by the symmetric determinant
\begin{equation} \det (z_1q_{ij}^{[1]} + \dots + z_4q_{ij}^{[4]}). \end{equation}
Let $\tilde S_x = p_x(\Xi)$  then, eliminating $(z_1:z_2:z_3:z_4)$ from the equations of $\Xi$, the equation of $\tilde S_x$ is a quartic form in $(x_1:x_2:x_3:x_4)$,  namely
\begin{equation}
\tilde S_x = \lbrace \det \ (q^{[k]}_j) = 0 \rbrace.
\end{equation}
Now consider $\tilde S \subset \mathbb P^3 \times \mathbb P^3$ and  its equations in $(x,y)$. Let $\tilde p_x: \tilde S \to \mathbb P^3$ be the first projection and $\tilde S_v = \tilde p_x(\tilde S)$.
Eliminating $y$, one computes that
\begin{equation}
\tilde S_v = \lbrace \det \ (q^{[k]}_j) = 0 \rbrace.
\end{equation}
Then $\tilde S_x = \tilde S_v$ and hence the surfaces $\tilde S_v$, $\tilde S$, $\tilde S_+$ are birational projective models of the same symmetroid $\tilde S_+$.
Moreover let $\tilde p_y: \tilde S \to \mathbb P^3$ be the second projection and $\tilde S_y = p_y(\tilde S)$, then $\tilde S_x$ and $\tilde S_y$ are projectively
isomorphic.  This follows because $\iota(\tilde S) = \tilde S$ and hence $\tilde p_y = \tilde p_x \circ \iota$. We keep the notation $\tilde S_v$ for 
$\tilde S_x$ and $\tilde S_y$. In particular $\tilde S_v$ is the birational image of $p_z: \tilde S_+ \to \mathbb P^3$. The next theorem summarizes our discussion and implements the picture.
\begin{theorem} The $K3$ surface $\tilde S$ in $\mathbb P^3 \times \mathbb P^3$ and the quartic symmetroid $\tilde S_+$ in $W$ are birational to the quartic surface $\tilde S_v$. Moreover 
this surface is the locus in $\mathbb P^3$ of the singular points of the singular quadrics of $W$.
  \end{theorem}
 Occasionally $\tilde S_v$ is said to be the \it Steinerian \rm of $\tilde S_+$, after some   tradition, \cite{DK} 7.2.    It is now the time to recall some facts on quartic double solids.
\subsection{\it Quartic double solids} \ \par
To begin we recall that a \it quartic double solid \rm is a finite double cover 
\begin{equation}
f: X \to \mathbb P^3
\end{equation}
whose branch scheme is a quartic surface $B \subset \mathbb P^3$. See \cite{C}, \cite{W} and \cite{CPS} for new results and update.
\it We assume that  no line is in $B$ and that $\Sing B$ is a finite set of ordinary double points. \rm Let us consider the blowing  up
 \begin{equation} \sigma: \mathbb P^{3'} \to \mathbb P^3, \end{equation} 
 of $\Sing B$, then  a desingularization of $X$ is provided by the base change \begin{equation}
\begin{CD} \label{MAINDIAG}
 {X'}  @>{f'}>> {\mathbb P^{3'}} \\
  @V{\sigma'}VV @V{\sigma}VV     \\
 X @>f>> {\mathbb P^3.} \\
\end{CD}
\end{equation}
Then $f'$ is the finite double cover branched on the strict transform of $B$ by $\sigma$. This is a smooth, minimal model of $B$ embedded in $\mathbb P^{3'}$, we denote by
\begin{equation}
B' \subset \mathbb P^{3'}.
\end{equation}
The line geometry of $\mathbb P^3$ strongly influences the geometry of $X$.  Consider indeed the universal line $\mathbb U = \lbrace (x, \ell) \in \mathbb P^3 \times \mathbb G \ \vert \ x \in \mathbb U_{\ell} \rbrace$ and its projections
\begin{equation}
\begin{CD}
{\mathbb P^3} @<{t}<< \mathbb U @>{u}>> \mathbb G. \\
\end{CD}
\end{equation}
Then $u$ is the projective universal bundle and $t$ its tautological morphism. The quartic surface $B$ clearly defines defines a rational section
\begin{equation}
s: \mathbb G \to \mathbb P(\Sim^4 U^{*}),
\end{equation}
sending the ray $\ell \in \mathbb G$ to the intersection divisor $ \mathbb U_{\ell} \cdot B \in \vert \mathcal O_{\mathbb U_{\ell}}(4) \vert$. 
\begin{definition} $s:  \mathbb G \to \mathbb P(\Sim^4 U^{*})$ is the section defined by $B$.
 \end{definition}
Since $B$ does not contain lines, the rational map $s$ is a morphism. Let 
 $$
  \mathbb P(\Sim^2 U^{*}) \subset \mathbb P(\Sim^4 U^{*}).
 $$
be the embedding defined by the squaring map. By definition this means that any point $d \in \vert \mathcal O_{\mathbb U_{\ell}}(2) \vert = \mathbb P(\Sim^2 U^{*})_{\ell}$ is embedded
as the point $$ 2d \in \vert \mathcal O_{\mathbb U_{\ell}}(4) \vert = \mathbb P(\Sim^4 U^{*})_{\ell}. $$  We use this embedding to define the family of bitangent lines to $B$.
\begin{definition} $\mathbb F(B)$ is the pull-back of $\mathbb P(\Sim^2 U^{*})$ by $s$. \end{definition}
Clearly, $\Supp \mathbb F(B)$ is the set of bitangent lines to $B$. The structure of $\mathbb F(B)$ is known, see \cite{W} 3 and \cite{CZ} 3.4. We summarize as follows.
\begin{theorem} If $B$ is general $\mathbb F(B)$ is a smooth integral surface and
$$
[\mathbb F(B)] = 12 \sigma_{1,1} + 28 \sigma_{2,0} \in CH^*(\mathbb G).
$$
\end{theorem}
In the classical language $\mathbb F(B)$ has \it order \rm $12$ and  \it class \rm $28$. These numbers are easily explained: $28$ is the number of bitangent lines to a general plane section of $B$, which is a smooth plane quartic. Instead $12$ is the number of ordinary nodes of the branch curve of a general projection $B \to \mathbb P^2$. \begin{definition} $\mathbb F(B)$ is the congruence of bitangent lines of $B$. \end{definition} 
The surprising case of $\mathbb F(B)$, $B$ a quartic symmetroid, will be discussed in detail.  Let us  fix our notation for the natural involutions of $X$ and $X'$.
\begin{definition} We respectively denote by $j': X' \to X'$ and by $j: X \to X$ the biregular involutions induced by $f': X' \to \mathbb P^{3'}$ and by $f: X \to \mathbb P^3$.
\end{definition}
 Finally we introduce \it the Fano surface of lines  \rm of $X'$. Let $\ell \in \mathbb F(B)$ then the line $\mathbb U_{\ell}$ is bitangent to $B$. Moreover, for $\ell$ general in $\mathbb F(B)$, it is also true  that 
 $\mathbb U_{\ell} \cap \Sing B = \emptyset$. Assuming this the curve ${f'}^*\mathbb U_{\ell}$ splits as follows:
\begin{equation}
{f'}^*\mathbb U_{\ell} =   R'_{\ell, +} + R'_{\ell, -}   
\end{equation}
where   $R'_{\ell, +}$, $ R'_{\ell,-}$    are biregular to $\mathbb P^1$ and exchanged by the involution $j'$. They belong to the Hilbert scheme of curves of arithmetic genus $0$ and of degree $1$ for $(f' \circ \sigma)^* \mathcal O_{\mathbb P^3}(1)$. This is a well known connected surface and the family of curves   $R'_{\ell, +}$, $ R_{\ell,-}$    is open and dense in it.   We denote it by $\mathbb F(X')$. 
\begin{definition} $\mathbb F(X')$ is the Fano surface of lines of $X'$. \end{definition}
The surface $\mathbb F(X')$ is \it smooth and irreducible \rm for a general $B$, \cite{W, C, CZ}. We point out that $j': X' \to X'$ defines a biregular involution 
 \begin{equation}
{j'}^*: \mathbb F(X') \to \mathbb F(X').
\end{equation}
As above let $\sigma: \mathbb P^{3'} \to \mathbb P^3$ be the blow up of $\Sing B$ and let  $\mathbb G'$ the Hilbert scheme  of the pull-back by $\sigma$ of a general line of $\mathbb P^3$.
Then the push-forward of cycles by the blowing up $\sigma$ defines a natural birational morphism
\begin{equation}
\sigma_*: \mathbb G' \to \mathbb G.
\end{equation}
Let $\mathbb F(B) \to \mathbb G$ be the inclusion map, then the base change by $\sigma_*$
\begin{equation} \begin{CD} { \mathbb F(B')} @>>> {\mathbb G'} \\
  @VVV @V{\sigma_*}VV  \\
  {\mathbb F(B)} @>>> {\mathbb G}\end{CD}
\end{equation}
defines the surface $\mathbb F(B')$ in $\mathbb G'$. If $\Sing B = \emptyset$ this is $\mathbb F(B)$, we only mention that $\mathbb F(B') \to \mathbb F(B)$ is the  normalization map if $B$ is general with $\Sing B \neq \emptyset$. Moreover the push-forward of cycles defines the next commutative diagram, where $f'_*$ generically concides with the quotient map of ${j'}^{*}$:
\begin{equation} \label{MAINDIAGFano}  \begin{CD} {\mathbb F(X') } @>{f'_*}>>{ \mathbb F(B')}   \\
@VVV  @V{\sigma_*}VV  \\
{\mathbb F(X)} @>{f_*}>> {\mathbb F(B)}.\end{CD}
\end{equation}
  To conclude we recall the following theorem. \begin{theorem} Let $X$ be a general quartic double solid then the map
$$ f'_*: \mathbb F(X') \to \mathbb F(B') $$ is an \'etale double covering of smooth regular surfaces of general type.   \end{theorem}
See the seminal papers \cite{C}, \cite{W} and \cite{CZ} for some recent new results. 
\begin{remark} \rm Clearly, for a general $B$, the morphism  is defined by a non zero $2$-torsion element of $\Pic \mathbb F(B')$. Then $\mathbb F(B')$ is not simply connected. However it is regular and of  general type.  \rm Coming to a \it quartic symmetroid $\tilde S_+$\rm, the surface $\mathbb F( \tilde S_+)$ shows up as   a very interesting     limit of a general $\mathbb F(B)$. We will see that it is singular and normalizes to the Reye congruence $S$. \end{remark}
 
 \section{The Artin-Mumford counterexample revisited}
\subsection{ \it Artin-Mumford double solids} \ \par
The study of the rationality problem for quartic double solids, with its related issues, plays a very important role, historically and not only.
The unirationality of a quartic double solid is known since longtime, cfr. \cite{B1}-4, \cite{R}-10, 11. Its irrationality, when it is branched on a quartic
symmetroid, was proven in 1972 by Artin and Mumford, \cite{AM}.  This result is one of the three first counterexamples to L\"uroth problem in dimension $3$, 
appearing simultaneously in 1971 - 1972 and relying on different
methods. The other examples also rely on very famous results: the proof of the irrationality of a smooth cubic threefold, by Clemens and Griffiths, and  the irrationality of a smooth quartic threefold, 
proven by Manin and Iskovskih, \cite{CG, IM}. Nowadays the rationality problem for quartic double solids is settled, by application of similar 
methods and further work, at least as follows, cfr. \cite{CPS}.
\begin{theorem} Let $f:X \to \mathbb P^3$ be a quartic double solid and let $m = \vert \Sing B \vert$. Then $X$ is irrational if $0 \leq m \leq 6$ and rational if
$m \geq 11$. \end{theorem}
 Nevertheless this matter is a not exhausted  field of top interest for several reasons. Though there is no space for further digression,
 let us mention once more the results on \it non stable rationality\rm, and \it non decomposition \rm of the diagonal in $H^*(X, \mathbb Z)$,  for a very general quartic double solid \cite{V1, V2}.
\begin{definition} We say that a quartic double solid is an Artin-Mumford double solid if its branch surface is a general quartic symmetroid $\tilde S_+$.
\end{definition}
 To treat Artin-Mumford double solids \it  let us fix our notation as follows: \rm as above $W \subset \mathbb Q$ is a \it general \rm web in the space $\mathbb Q$ of quadric surfaces of $\mathbb P^3$.
Then $W$ is a $3$-dimensional subspace and $\tilde S_+ = \mathbb W \cdot \mathbb Q^3$. We denote by
\begin{equation}
f: \tilde W \to W
\end{equation}
the finite double covering of $W$ branched on $\tilde S_+$. The next diagram   defines,    exactly as in (\ref{MAINDIAG}) of the previous section, a desingularization $\tilde W'$ of $\tilde W$:
\begin{equation}
\begin{CD} \label{MAINDIAGW}
 {\tilde W'}  @>{f'}>> {W'} \\
  @V{\sigma'}VV @V{\sigma}VV     \\
 {\tilde W} @>f>> {W.} \\
\end{CD}
\end{equation}
We will say, with a slight abuse,  that the morphism $f': \tilde W' \to  W'$ is the desingularization of $f: \tilde W \to W$ and that $\tilde W'$ is the smooth model of $\tilde W$. \par
In the above diagram, constructing an Artin-Mumford double solid,  the implicit presence of an Enriques surface is clear. $\tilde S$ is indeed  the minimal desingularization of $\tilde S_+$ and an \'etale double covering of the Reye congruence $S$. The existence of $S$ is not 
mentioned in Artin-Mumford paper \cite{AM}. However the irrationality of $\tilde W$ follows there from the same irrationality feature of $S$, namely the presence of non zero torsion in the third cohomology group. At first it is shown in \cite{AM}  that the torsion subgroup of it is birationally invariant, for any smooth projective variety. Then  the next theorem is proven.
\begin{theorem} The torsion of $H^3(\tilde W', \mathbb Z)$ is non trivial. \end{theorem}
The proof relies on the notion of Brauer group and on some Severi-Brauer varieties, conic bundles in this case, related to $\tilde W$.  Since $H^3(\mathbb P^3, \mathbb Z) = 0$, the irrationality of $W$ follows. 
Moreover, as we will see, $\tilde W$ is unirational. Then  $\tilde W$ is a counterexample to L\"uroth problem.  Since the above torsion group is a stably rational invariant, $\tilde W$ is not stably rational as well, \cite{V1}. \par  Actually both $S$ and $\tilde W'$ have a non zero torsion subgroup in the third cohomology and this is, in both cases, $\mathbb Z/2\mathbb Z$.  The existence of $S$ is variably used and considered in the literature on Artin-Mumford double solids. The same is even more true for an explicit and geometric description of the relation between $H^3(S, \mathbb Z)$ and $H^3(\tilde W', \mathbb Z)$. As far as we know, the surface $S$ was used at first by Beauville in \cite{B1}-9, to this respect,  as follows.  \par Let $\tilde { \mathbb G}$ be the blow up of  $\mathbb G$ at $S$, one constructs very geometrically a dominant morphism $\tilde \phi: \tilde {\mathbb G} \to \tilde W$. Since $\mathbb G$ is rational, $\tilde W$ is unirational. Moreover $H^*(S, \mathbb Z)$ is a summand of
$H^*(\tilde {\mathbb G}, \mathbb Z)$ and hence $\tilde \phi$ naturally defines a  homomorphism $h: H^*(\tilde W', \mathbb Z) \to H^*(S, \mathbb Z)$. Using it, a proof of the previous theorem is given by application of cohomological methods, see \cite{B1} p.30. In particular $h$ restricts to an isomorphism between the torsion subgroup of $H^3(\tilde W', \mathbb Z)$ and $H^3(S, \mathbb Z) \cong \mathbb Z/2 \mathbb Z$. See also \cite{IKP}-4. \par Next we introduce some geometry  linking  $S$ and $\tilde W'$. From it a quite   explicit     description of the torsion of $H_3(\tilde W', \mathbb Z)$,  via $H_1(S, \mathbb Z)$,
will follow.  Let $\mathbb U \vert S$ be the universal line $\mathbb U \to \mathbb G$ restricted over $S$. The description essentially relies on a rational map $\upsilon: \mathbb U \vert S \to \tilde W'$ embedding a general fibre of $\mathbb U \vert S$ as  a line of $\tilde W'$, that is, an element of the Fano surface $\mathbb F(\tilde W')$. Then a 'cylinder map' $\mathsf c: H_1(S, \mathbb Z) \to H_3(\tilde W', \mathbb Z)$ is induced by $\upsilon$. As we will see $\mathsf c$ is injective.  Hence $H_1(S, \mathbb Z) = \mathbb Z/2 \mathbb Z$ injects in $H_3(\tilde W', \mathbb Z)$ and,  by Poincar\'e duality, in $H^3(\tilde W', \mathbb Z)$. Actually $\mathsf c$ is an isomorphism. \par
 Let us also point out that $\mathbb U \vert S$ defines a morphism onto the surface of bitangent lines to $\tilde S_+$, sending $\ell \in S$ to the pencil $\mathbb I_{\ell}$ of all quadrics of $W$ through   $\mathbb U_{\ell}$. This is a bitangent line to $\tilde S_+$. We denote this morphism by
\begin{equation}
\nu: S \to \mathbb F(\tilde S_+)  
\end{equation}
  where $\mathbb F(\tilde S_+)$ is defined like $B$ in diagram (\ref{MAINDIAGFano}).    Differently from a general quartic, $\mathbb F(\tilde S_+)$ is non normal and birational to the Reye congruence $S$. We outline here a description of $\mathbb F(\tilde S_+)$. However, during our work, we became aware of the beautiful description already performed by Ferretti,  \cite{F1}, \cite{F2}-3.
 \begin{remark} \rm The latter one relates $\mathbb F(\tilde S_+)$ to the theory of EPW-sextics, when the corresponding hyperkaehler fourfold is the Hilbert scheme $\tilde S^{[2]}$ of two points of $\tilde S$. The surface $\mathbb F(\tilde S_+)$ is birational to the locus of fixed points of the natural involution induced on $\tilde S^{[2]}$ by the quartic $\tilde S_+$, \cite{F2}-3.1.1.
 \end{remark} 
\subsection{\it The congruence of bitangent lines $\mathbb F(\tilde S_+)$} \ \par
As in (\ref{QUADRICS}) let $\mathbb Q$ be the space of quadrics of $\mathbb P^3$, we assume that $W \subset \mathbb Q$ is a general web. Then $\tilde S_+$ is a general quartic 
symmetroid. Since now
\begin{equation}
\mathbb G_{\mathbb Q} \subset \mathbb P^{44}
\end{equation}
is the Pl\"ucker embedding of the Grassmannian of lines of $\mathbb Q$, then $\mathbb G_Q$ is the family of all pencils of quadrics of $\mathbb P^3$. The space of its orbits, under the action of $\Aut \mathbb P^3$, classifies these pencils up to projective equivalence. The classification goes back at least to Corrado Segre, \cite{S1, S2}. See \cite {FMS}, \cite{AEI}-7 for some recent revisiting. It follows from the classification  that the locus 
\begin{equation} \mathbb F(\mathbb Q^3) \subset \mathbb G_{\mathbb Q}, \end{equation} of pencils which are bitangent lines to the quartic discriminant $\mathbb Q^3 \subset \mathbb Q$, is an irreducible subvariety of codimension two.  Its description is classical and very well known: a \it general \rm element $P$ of $\mathbb F(\mathbb Q^3)$ is a pencil of quadrics whose base scheme is a complete intersection of two quadrics
$$
L \cup C \subset \mathbb P^3,
$$
where $L$ is a line and $C$ is a smooth,  rational normal cubic curve. Moreover
$$
L \cap C = \lbrace v_1, v_2 \rbrace =  \Sing L \cup C
$$
where $ v_1, v_2$ are ordinary nodes of $L \cup C$. Notice that $P$ is generated by two quadrics $Q_1$ and $Q_2$ of rank $3$, respectively singular at
$v_1$ and $v_2$. These are the tangency points of $P$ to $\mathbb Q^3$ and the two singular quadrics of the pencil. \par
More globally $\mathbb F(\mathbb Q^3)$ is parametrized by the smooth correspondence
\begin{equation}
\mathcal G := \lbrace (\ell, P) \in \mathbb G \times \mathbb G_{\mathbb Q} \ \vert \ \mathbb U_{\ell} \ \text  {is in the base scheme of $P$} \rbrace.
\end{equation}
 Indeed consider the natural projections 
\begin{equation} 
\begin{CD}
{\mathbb G} @<{u_{\mathcal G}}<< {\mathcal G} @>{t_{\mathcal G}}>> {\mathbb F(\mathbb Q^3)}, \\
\end{CD}
 \end{equation}
then $u_{\mathcal G}: \mathcal G \to \mathbb G$ is smooth and a Grassmann bundle whose fibre at $\ell$ is the Grassmannian of the pencils of quadrics containing $\mathbb U_{\ell}$. Moreover it is easy to see that $t_{\mathcal G }: \mathcal G \to \mathbb G_{\mathbb Q}$ is a birational morphism onto its image $\mathbb F(\mathbb Q^3)$. \par
Clearly $t_{\mathcal G}$ is biregular over each $P \in \mathbb F(\mathbb Q^3)$ such that $P$ is a general pencil as above. Notice also that, for any $P \in \mathbb F(\mathbb Q^3)$,
the fibre of $t_{\mathcal G}$ at $P$ is a scheme supported on the points $(\ell, P) \in \mathbb G \times \lbrace P \rbrace$ such that $\mathbb U_{\ell}$ is in the base scheme of $P$.
This remark is the starting point to describe $\Sing \mathbb F(Q^3)$. 
\begin{theorem} $\Sing \mathbb F(\mathbb Q^3)$ is irreducible of codimension $3$ in $\mathbb G_{\mathbb Q^3}$. Let $P$ be general in $\Sing \mathbb F(\mathbb Q^3)$,
then $P$ is an non normal  ordinary double point and the base scheme $B_P$ of $P$ contains a conic $D = L \cup L'$ of rank $2$.
\end{theorem}
Actually the pairs $(L,P), (L',P) \in \mathcal G$ correspond to the branches of the node $P \in \mathbb F(\mathbb Q^3)$. Then such a general singular
point $P$ is a pencil whose base scheme is a complete intersection $D \cup D'$, $'$ being a smooth conic. This implies that $P$ contains a quadric $Q$ of rank $2$, namely $Q \in \mathbb Q^2$
is the union of the planes supporting $D$ and $D'$. In particular let us consider
\begin{equation}
\Sigma_{\mathbb Q^3} \subset \mathbb G_{\mathbb Q},
\end{equation}
the locus of all pencils $P$ intersecting $\mathbb Q^2$, then the next theorem follows.
 \begin{theorem} $ \ \Sing \mathbb F(\mathbb Q^3) \ = \ \mathbb F(\mathbb Q^3) \cdot \Sigma_{\mathbb Q^2}$.
\end{theorem}
We use these properties for a basic description of the surface $\mathbb F(W)$. Let 
\begin{equation}
\mathbb G_W \subset \mathbb G_{\mathbb Q}
\end{equation}
be the Grassmannian of lines of $W$, then $\mathbb G_W$ is a Schubert variety and a smooth $4$-dimensional quadric in $\mathbb G_{\mathbb Q}$. According to the classification of pencils
of quadrics, the space $\mathbb F(\mathbb Q^3)$ is quasi-homogeneous for the action of $\Aut \mathbb P^3$, then union of finitely many orbits $\Gamma$. By transversality of general 
translate,  we can assume that $\mathbb G_W$ is transversal to each $\Gamma$. Let 
\begin{equation}
\mathsf O \subset \mathbb F(\mathbb Q^3) \label{OPENORB}
\end{equation}
be the family of pencils $P \in \mathbb F(\mathbb Q^3)$ which are general as above, it is
easy to see that $\mathsf O$ is irreducible and the unique orbit which is open in $\mathbb F(\mathbb Q^3)$. By transversality of general translate again,  $\mathsf O \cap \mathbb G_W$  
is a smooth, irreducible open set of the surface $\mathbb F(\tilde S_+)$. From Segre classification for pencils $P$ of $\mathbb F(\mathbb Q^3)$ one can see the complement of $\mathsf O \cap \mathbb F(\tilde S_+)$
in $\mathbb F(\tilde S_+)$. This is a union of curves in $\mathbb F(\tilde S_+)$. For it we fix our notation and description as follows:  
\begin{equation} \label{CURVES N AND D}
\mathbb F(\tilde S_+) - (\mathsf O \cap \mathbb F(\tilde S_+)) = N \cup D.
\end{equation}
  (1) {  \sf The singular curve $ N := \Sing \mathbb F(\tilde S^+)$.}  The above properties imply that $N$ is the family of bitangent lines intersecting $\Sing \tilde S_+$. Therefore we have $ N = \cup_{o \in \Sing \tilde S_+} N_o$, where the curve $N_o$ is
\begin{equation}
N_o := \lbrace P \in \mathbb F(\tilde S_+) \ \vert \ o \in P \rbrace.
\end{equation}
The curve $N_o$ lies in the plane $\mathbb P^2 \subset \mathbb G_W$ parametrizing all rays passing through $o$. As is well known $N_o$ is union of two smooth cubics intersecting at nine points, three of which are on a smooth conic. Moreover $N_o$ is the discriminant curve of the conic bundle structure naturally defined by
$$
\pi_o \circ f': \tilde W' \dasharrow \mathbb P^2,
$$
where $\pi_o: W \to \mathbb P^2$ is the linear projection of center $o$, cfr. \cite{DK}-7, \cite{IKP}-4.
 \par
 (2) {\sf The curve $D$ of hyperflex bitangent lines.}  We say that $P$ is a hyperflex tangent line if $P \cdot \tilde S_+$ is a unique point of $P$, with multiplicity $4$. We omit more details on the well known curve
 $D$, since we will not use it.
 \medskip \par 
  In what follows we will concentrate on two important and elementary peculiarities of Artin-Mumford double solids, which are determinant for the
 existence of a $\mathbb Z/2 \mathbb Z$ torsion subgroup of $H^3(\tilde W', \mathbb Z)$. These are: 
 \begin{itemize}
 \item A rational map $\psi: \mathbb G \dasharrow W$ factorizing as follows
 \begin{equation}
 \begin{CD}
 {\tilde {\mathbb G}} @>{\tilde \phi}>> {\tilde W} \\
 @V{\beta}VV @VfVV \\
 {\mathbb G} @>{\psi}>> W, \\
 \end{CD}
 \end{equation}
 where $\beta: \tilde G \to \mathbb G$ is the blowing up of $S$ and $\psi \circ \beta$ is a morphism. A remarkable fact is that $f$ is the Stein factorization of 
 $\psi \circ \beta$.
 \item A birational  morphism $  n:    S \to \mathbb F(\tilde S_+)$ and its birational lifting
  $$
    \tilde n:    S \to \mathbb F(\tilde W'),
  $$
   onto its image $\mathbb F^+ \subset \mathbb F(\tilde W')$. Remarkably the surface $\mathbb F(\tilde W')$ splits as 
   \begin{equation}
   \mathbb F(\tilde W'   ) = \mathbb  F^+ \cup \mathbb F^-, 
   \end{equation} with $\mathbb F^+, \mathbb F^-$ birational to $S$ and exchanged  by the   map ${j'}^*$,  see \ref{WIRTINGER}.   
   \end{itemize}
   \subsection{ \it The  rational map $\psi: \mathbb G \dasharrow W$}     \ \par
 Let $W$ be a general web as above, we consider the rational map
 \begin{equation}
\psi: \mathbb G \dasharrow W,
\end{equation}
sending a general $\ell \in \mathbb G$ to the unique quadric $Q \in W$ through  $\mathbb U_{\ell}$. We recall that the Reye congruence $S$ of $W$ is by definition the degeneracy scheme of the map of vector bundles in (\ref{EVALUATION}). This easily implies that $S$ is the indeterminacy scheme of $\psi$. To have a resolution of it we study the graph 
 \begin{equation}
\tilde {\mathbb G} := \lbrace (\ell,  Q) \in \mathbb G \times W \ \vert \  \mathbb U_{\ell} \subset Q \rbrace,
\end{equation}
of $\psi$. We also consider its two natural projections
\begin{equation} \label{DIAGRAM}
\begin{CD}
{\mathbb G} @<{\beta}<< {\tilde {\mathbb G}} @>{\phi}>> W.\\
\end{CD}
\end{equation}
From the description of $S$ in (\ref{REYEGEOM}) it follows that $\beta$ is the contraction to $\mathbb G$ of a $\mathbb P^1$-bundle over $S$. Then $\beta$ is the blowing up of $\mathbb G$ at  $S$.  Notice also that the linear system of the quadrics of $W$ through $\mathbb U_{\ell}$ is the fibre $\beta^*(\ell)$ of $\beta$. For the exceptional divisor of $\beta$ 
we fix the notation
\begin{equation} 
\mathbb I := \lbrace (\ell, Q) \in \mathbb G \times W \ \vert \ \dim \beta^*(\ell) = 1 \rbrace. \label{EXCEPTIONAL}
\end{equation}
Clearly the birational morphism $\beta$ restricts on $\mathbb I$ to the $\mathbb P^1$-bundle 
$$ \beta \vert \mathbb I: \mathbb I \to S $$
and, at $\ell \in S$, its fibre $\mathbb I_{\ell}$ is the pencil of quadrics of $W$ through $\mathbb U_{\ell}$. 
 Let $w \in W$ then its corresponding quadric embedded in $\mathbb P^3$ will be denoted by
\begin{equation}
Q_w \subset \mathbb P^3.
\end{equation}
  Now we consider the projection map $\phi: \tilde{\mathbb G} \to W$. This is a morphism whose general fibre is not connected.  
Indeed let $w \in W$ then we have
\begin{equation}  \phi^*(w) = \tilde {\mathbb G} \cdot (\mathbb G \times \lbrace w \rbrace). \end{equation}
The equality just says that $\phi^*(w)$ is the Hilbert scheme of lines of the quadric surface $Q_w$. Then the proof of the next theorem is elementary.
\begin{theorem} For a general $w \in W$ the two rulings of lines of $Q_w$ are the connected components of $\phi^{*}(w)$. For any fibre $\phi^*(w)$ can be as follows:  \par 
(i) $\rank Q_w = 4$: disjoint union of two smooth conics, \par (ii) $\rank Q_w = 3$: a smooth conic having multiplicity $2$,  \par (iii) $\rank Q_w = 2$: union of two planes $P_1, P_2$. $P_1 \cap P_2 = \lbrace \text {one point }\rbrace$. \end{theorem}
Clearly $\phi^*(w)$ has type (iii) iff $w \in \Sing \tilde S_+$ and (ii) iff $w \in \tilde S_+ - \Sing \tilde S_+$. The type is (i) and the fibre is not connected iff
$w \in \mathbb P^3 - \tilde S_+$. \par Passing to the \it Stein factorization of $\phi$ \rm we have the diagram
\begin{equation} \label{RATVAR}
\begin{CD}
{\tilde {\mathbb G}} @>{\tilde \phi}>> {\tilde W} @>f>> W,\\
\end{CD}
\end{equation}
where $\phi = \tilde \phi \circ f$ and $f$ is a finite double covering. It is clear that 
\begin{equation}  f: \tilde W \to W \end{equation} is branched on the symmetroid $\tilde S_+$ of $W$ and defines the Artin-Mumford double solid $\tilde W$.  Let $w \in W$ then the fibre
$f^*(w)$ is finite of length two. To denote the points of $\Supp f^*(w)$ we fix the following convention.
\begin{equation} \label{RULINGS}  \Supp f^*(w) := \lbrace w^+, w^- \rbrace, \end{equation}
moreover $w^+ = w^-$ iff $w \in \tilde S_+$. Clearly $w^+$ and $w^-$ label the rulings of lines of the quadric $Q_w$ if its rank is $\geq 3$. Let us roughly summarize as follows. \par  
\begin{remark} \rm $\tilde W$ is a parameter space for pairs $(w, w^+), (w, w^-)$, where $w \in W$ and $w^+, w^-$ are the rulings of $Q_w$. $f$ is the natural forgetful map.
\end{remark}
Finally we see in (\ref{RATVAR})  that a rational variety, namely $\tilde {\mathbb G}$, dominates $\tilde W$.   Hence $\tilde W$ is unirational and     the next theorem follows, cfr. \cite {B2}-6.
 \begin{theorem} $\tilde W'$ is unirational. \end{theorem}
  \begin{remark} \rm  Notably, the properties of a blowing up imply
\begin{equation}
H^*(\tilde {\mathbb G}, \mathbb Z) \cong H^*(\mathbb G, \mathbb Z) \oplus \sigma^* H^*(S, \mathbb Z),
\end{equation}
where $\sigma: \tilde {\mathbb G} \to \mathbb G$ is the blowing up of $\mathbb G$ at $S$.  That has its importance.\end{remark}
\subsection {\it $S$ and the Fano surface $\mathbb F(\tilde W')$} \ \par
Now we want to see that   $S$ strongly interacts    with the Fano surface $\mathbb F(\tilde W')$. Actually this is union of two irreducible components, birational to $S$. We begin
with the Grassmannian $\mathbb G_W \subset \mathbb G_{\mathbb Q}$ of lines of $W$ and fix the notation
\begin{equation}
\delta: S \to \mathbb G_W
\end{equation}
for the following rational map. For $\ell \in S$ let $\mathbb I_{\ell} \subset W$ be the linear system of quadrics through $\mathbb U_{\ell}$. For each $\ell \in S$ this is a pencil, defining
a point of $\mathbb G_W$. By definition this is $\delta(\ell)$ and, moreover, $\delta: S \to \mathbb G_W$ is a morphism. 
 \begin{theorem} $\delta: S \to \mathbb G_W$ is birational onto its image, moreover we have
 $$
 \delta(S) = \mathbb F(\tilde S_+).
 $$ 
 \end{theorem}
 \begin{proof} Recall that $W$ is always general, so that $\mathbb G_W$  is transversal to $\mathbb F(\mathbb Q^3)$. Now consider
 the orbit $\mathsf O \subset \mathbb F(\mathbb Q^3)$ as in  (\ref{OPENORB}). As observed this is the unique irreducible open orbit under the action of 
 $\Aut \mathbb P^3$ on $\mathbb F(\mathbb Q^3)$. Moreover  $\mathsf O \cap \mathbb F(\tilde S_+)$ is a smooth, irreducible open set of the surface $\mathbb F(\tilde S_+)$.  
 Finally a point $P$ of it is a pencil whose base scheme $B_P$ is $L \cup C \subset \mathbb P^3$, where $C$ is a smooth rational normal cubic and $L$ is a line. Let $\ell \in \mathbb G$
 be the point defined by $L$, then we have $P = \mathbb I_{\ell}$ and $\delta (\ell) = P$. Since the only line in $B_P$ is $L$,  $\delta^{-1}(P) = \lbrace \ell \rbrace \subset S$. 
Moving $P$ along the mentioned open set of $\mathbb F(\tilde S_+)$ it follows that $\delta: S \to \mathbb F(\tilde S_+)$ is invertible and dominant.
 \end{proof}

  \subsection{\it The counterexample of Artin-Mumford} \ \par
 Finally we want to show that $H^3(\tilde W', \mathbb Z)$ has a nonzero $2$-torsion element, reconstructing it from the nonzero element of $H_1(S, \mathbb Z)$.     Let $\pi_1(S)$ be the fundamental group of $S$ and $\mathring S \subset S$ a non empty Zariski open set, we recall that the inclusion $i: \mathring S \to S$ defines a surjective homomorphsim
 $$ i_*: \pi_1(\mathring S) \to \pi_1(S). $$  This is well known for any complex, irreducible algebraic variety, \cite{FL} 0.3. Since we have  $H_1(S, \mathbb Z) = \pi_1(S) = \mathbb Z/2 \mathbb Z$,  then 
 the next property follows.
 \begin{proposition} \label{FUNDAMENTAL} For the generator $[\gamma]$ of $H_1(S, \mathbb Z)$ one can choose $\gamma$ so that its image is  in $\mathring S$. Then $\gamma$ defines a nonzero 
$2$-torsion element of $H_1(\mathring S, \mathbb Z)$. \end{proposition}
After this remark we consider the exceptional divisor $\mathbb I$ of the blowing up of $\mathbb G$ at $S$. As in (\ref{EXCEPTIONAL}) this is a $\mathbb P^1$-bundle $\beta: \mathbb I \to S$. Let $\mathbb I_{\ell}$ be its fibre at $\ell \in S$, then $\mathbb I_{\ell}$ sits in $\lbrace \ell \rbrace \times W = W$ as the pencil of quadrics 
\begin{equation}
\mathbb I_{\ell} = \lbrace Q \in W \ \vert \ \text{$Q$ contains the line of $\mathbb P^3$ defined by $\ell$} \rbrace.
\end{equation}
We already mentioned that the surface $\mathbb F(\tilde S_+)$, of bitangent lines to the quartic symmetroid $\tilde S_+$, just coincides with  the family of pencils 
\begin{equation}
 \lbrace \mathbb I_{\ell} \subset W, \ \ell \in S \rbrace.
\end{equation}
As a surface in the Grassmannian $\mathbb G_W$ of lines of $W$, this is the birational image of the morphism $\delta: S \to \mathbb G_W$, sending $\ell$ to $\mathbb I_{\ell}$.  It is a singular model of $S$ having class $12 \sigma_{20} + 28\sigma_{11}$ in $\mathbb G_W$.   Let
$ f: \tilde W \to W $ be the finite double cover branched on $\tilde S_+$. We also know that the fibre of $f$ at $Q \in W$ is bijective to the set of connected components of the family of lines of $Q$.

\begin{proposition} Let $\ell \in S$ then its inverse image $f^{-1}(\mathbb I_{\ell})$ is union of two or one  irreducible components, biregular to $\mathbb I_{\ell}$ via $f$. \end{proposition}
\begin{proof} Let $Q \in \mathbb I_{\ell}$ then the connected components of its family of lines are distinguished by the property of either containing the point $\ell$ or not. This easily implies that $f^{-1}(\mathbb I_{\ell})$ is the union prescribed by the statement.  \end{proof}
Let $w \in W$, we keep our notation $Q_w$ for the corresponding quadric in $\mathbb P^3$. For the irreducible components of $f^{-1}(\mathbb I_{\ell})$ we will have the notation
\begin{equation}
f^{-1}(\mathbb I_{\ell}) := R^+_{\ell} \cup R^-_{\ell},
\end{equation}
where $R^+_{\ell}$ is the ruling of $\ell$ in $Q_w$. Notice that no line is in the symmetroid of a general $W$. Then in this case we have  $R^+_{\ell} \neq R^-_{\ell}$, $\forall \ \ell \in S$.   
The above equality implies that the Enriques surface $S$ parametrizes an irreducible component of the Fano surface $\mathbb F(\tilde W)$ of lines $\tilde W$. Let
$ j: \tilde W \to \tilde W$ be the involution associated to $f$, then the next property easily follows. 
\begin{proposition} $\mathbb F(\tilde W)$ is union of two irreducible components, birational to $S$ and exchanged by the involution $j^*: \mathbb F(\tilde W) \to \mathbb F(\tilde W)$. \end{proposition}
\begin{remark}[\it A Wirtinger construction] \rm \label{WIRTINGER} To close on $\mathbb F(\tilde W)$, the following summary and concluding remark are due. Let $\mathbb F^{\pm} :=  \lbrace R^{\pm}_{\ell}, \ \ell \in S \rbrace$ then
\begin{equation}
\mathbb F(\tilde W) = \mathbb F^+ \cup \mathbb F^-, 
\end{equation} 
where $\mathbb F^+$ and $\mathbb F^-$ are biregular to the surface $ \mathbb F(\tilde S_+)$ of bitangent lines to the quartic
symmetroid $\tilde S_+$. Clearly its normalization is the disjoint union $S' \vee S''$ of two copies of $S$ and $j^*$ lifts to
the standard involution $$ j^{\vee}: S' \vee S'' \to S' \vee S'',$$ defined  by exchanging $S'$ with $S''$ via the identity map of $S$.
Moreover the pair $(\mathbb F(\tilde W), j^*)$ is constructed from the pair $(S' \vee S'', j^{\vee})$ by taking a suitable quotient
surface $\mathbb F(\tilde W) = S' \vee S'' / \sim$ of $S' \vee S''$. Let us consider a flat family $$\lbrace f_t: X_t \to \mathbb P^3, t \in T \rbrace, $$ 
deforming $f: \tilde W \to W$, of quartic double solids. Then the corresponding family   $\lbrace j^*_t: \mathbb F(X_t) \to \mathbb F(X_t), t \in T \rbrace$ of fixed point free involutions   
is a deformation of $j^*: \mathbb F^+ \cup \mathbb F^- \to \mathbb F^+ \cup \mathbb F^-.$ A similar deformation, for curves with a fixed point free involution,  is known as  \it a Wirtinger construction. \rm \qed
  \end{remark}
We can think of a point of $\mathbb I$ as a pair $(\ell, w) \in S \times W$ such that $Q_w$ contains the line $\mathbb U_{\ell}$. Moreover we can think of a point of $\tilde W$ as a pair $(w, r) \in W \times \tilde W$, where $r$ is a connected component of the family of lines in $Q_w$. 
\begin{definition} The fundamental morphisms are the maps
\begin{equation}
\upsilon^+: \mathbb I \to \tilde W \ , \ \upsilon^-: \mathbb I \to \tilde W
\end{equation}
defined as follows: $\upsilon^{+}(w, \ell) := (w,w^{+})$, $w^+$ being the connected component of $\ell$ in the family of lines of $Q_w$. Moreover the map $w^-$ is $j^* \circ \upsilon^+$. 
\end{definition}
 
\begin{theorem}  The fundamental morphisms have degree $6$. \end{theorem}
\begin{proof} We prove the statement for $\upsilon^+$, the case of $\upsilon^-$ is similar. The degree of $\upsilon^+$ is the number of pairs $(\ell', w')$ such that $w = w'$ and $w^+ = r$, where
$(w, r) = \upsilon^+(\ell, w)$ and $(\ell, w)$ is general in $\tilde W$. Then we can assume $\rank Q_w = 4$ and $\ell$ general in $S$. To compute this number consider the restriction of $W$ to $Q_w$: this is a net $N$ in $ \vert \mathcal O_{Q_w}(2) \vert := \mathbb P^8$.  Hence $N$ is a plane in $\mathbb P^8$ and  the degree of $\upsilon^+$  is the number of reducible elements of $N$, containing a line of the ruling of  $\ell$.  Let $\vert L \vert$ be such a ruling and $H \in \vert \mathcal O_{Q_w}(1) \vert $, we have $\vert L \vert := \mathbb P^1$ and $\vert 2H - L \vert := \mathbb P^5$ is a linear system of rational cubics. Let  
$$s: \mathbb P^1 \times \mathbb P^5 \to \mathbb P^8, $$
be the sum map, sending $(L,C) \in \mathbb P^1 \times \mathbb P^5$ to $L+C \in \vert 2H \vert = \mathbb P^8$. We claim that $s^* N$ is finite. Then its length is  the degree of $\upsilon^+$ and coincides
with the degree of the Segre embedding of $\mathbb P^1 \times \mathbb P^5$, which is six. To prove our claim, assume $s^*N$ is not finite. Then each element of
$\vert L \vert$ is a point of  the Enriques surface $S \subset \mathbb G$ and the union of these points is a conic. Since
$\ell$ is general in $S$ then $S$ is uniruled: a contradiction.
 \end{proof}
 Finally let $\sigma': \tilde W' \to \tilde W$ be the natural desingularization obtained, as in (\ref{MAINDIAGW}), by blowing up the set of ten nodes $\Sing \tilde W$. Then the diagram
\begin{equation}
\begin{CD}
S @<{\beta}<< {\mathbb I} @>{\upsilon^+}>> {\tilde W} @<{\sigma'}<<{\tilde W'} \\
\end{CD}
\end{equation}
defines a 'cylinder map'
\begin{equation}
\label{CYLINDER}
\mathsf c:  H_1 (S, \mathbb Z) \to H_3(\tilde W', \mathbb Z),
\end{equation}
where $\mathsf c := \sigma'^{*} \circ \upsilon^+_*  \circ \beta^* $. Now let us consider the non empty Zariski open set $\mathring S \subset S$ of points $\ell \in S$ satisfying the following conditions: 
\begin{itemize} \it
\item[(a)] $\upsilon^+: \mathbb I \to \tilde W$ is finite over $\upsilon^+(\mathbb I_{\ell})$ and generically unramified at $\mathbb I_{\ell}$, \par 
\item[(b)] $\upsilon^+(\mathbb I_{\ell} ) \cap \Sing \tilde W = \emptyset$, that is, $\rank Q_w \geq 3$,  $\forall (w, r) \in \upsilon^+(\mathbb I_{\ell})$, \par
\item [(c)] $\delta (\ell) \notin \Sing \delta(S)$, where the map  $\delta: S \to \mathbb G_W$ sends $\ell$ to $\mathbb I_{\ell}$.  \par
\end{itemize} 
Let $[\gamma] \in H_1(S, \mathbb Z) = \mathbb Z / 2 \mathbb Z$ be the generator of this group. By proposition \ref{FUNDAMENTAL}  we can choose $\gamma$ so that its image is in the Zariski open set 
$\mathring S$. Now,  in the euclidean topology, $\mathbb I$ is homeomorphic to $S \times \mathbb P^1$ and we have   \begin{equation} H_3(\mathbb I, \mathbb Z) \cong H_1(S, \mathbb Z) \otimes H_2(\mathbb P^1, \mathbb Z) \cong \mathbb Z / 2\mathbb Z.
 \end{equation}
Indeed, one has $H_1(\mathbb P^1, \mathbb Z) = H_3(\mathbb P^1, \mathbb Z) = 0$. Moreover, by Poincar\'e duality, one has $H_3(S, \mathbb Z) \cong H^1(S, \mathbb Z)$. Furthermore $H^1(V, \mathbb Z)$ has no torsion for any smooth, projective variety $V$: see \cite{V1} proof of lemma 2.9 or deduce this property from the exponential sequence of $V$ . Since  the first Betti number of $S$ is zero, it follows $H^1(S, \mathbb Z) = 0$. Then,  applying K\"unneth formulae to $\mathbb I$, the above isomorphism follows.  Now let us consider on $\mathbb I$ the $3$-cycle 
$$ \beta^* \gamma := \mathbb I_{\gamma}. $$
Clearly $\mathbb I_{\gamma}$ defines the $2$-torsion class generating  $H_3(\mathbb I, \mathbb Z)$ and the $3$-cycle $$ T_{\gamma} := \upsilon_*^+(\mathbb I_{\gamma}). $$ By (b) $T_{\gamma}$ is a $3$-cycle of $\tilde W - \Sing \tilde W$, moreover $\sigma'$ is biregular over the set $\tilde W - \Sing \tilde W$. Then consider the $3$-cycle 
$ T'_{\gamma} := {\sigma'}^{*}(T_{\gamma}) $ of $\tilde W'$ and its class
  \begin{equation}
 \tau' = \mathsf c([\gamma]) \in H_3(\tilde W', \mathbb Z).
 \end{equation}
 This is a $2$-torsion element, we can now conclude via the following result.
 \begin{theorem} $\tau'$ is a nonzero element. \end{theorem}
\begin{proof} It suffices to show that the image of $\tau'$ by $(\upsilon^{+*} \circ \sigma'_*)$ is nonzero in $H_3(\mathbb I, \mathbb Z)$. This is equivalent to show that
 $$
 (\upsilon^{+*} \circ \sigma'_* \circ { \sigma'}^{*}) (\tau') = \upsilon^{+*}(\tau') = [\mathbb I_{\gamma}].
 $$
 Since $\sigma'$ is biregular over $\tilde W - \Sing \tilde W$ the first equality is immediate. Let us prove the second one. For the
 class $[\gamma]$ of $H_1(\mathring S, \mathbb Z)$ we can even assume that the map $\gamma: (0,1) \to \mathring S$ is a real analytic embedding. We have 
 $$\upsilon^{+*} T'_{\gamma} = m\mathbb I_{\gamma} + Z, $$ 
 where $Z$ is the image by $i_*$ of a cycle of $\mathbb I - \mathbb I_{\gamma}$ and $i: (\mathbb I - \mathbb I_{\gamma}) \to \mathbb I$ is the inclusion. Let $\Gamma := \gamma([0,1])$ and $\ell \in \Gamma$ then, by assumption (a) on $\gamma$, the morphism $\upsilon^+$ is finite over $\upsilon^+(\mathbb I_{\ell})$ and generically unramified along $\mathbb I_{\ell}$. This implies $m = 1$ and that $Z$ is a $3$-cycle. Now let $t \in S - \Gamma$, we claim that $\upsilon^+( \mathbb I_t)$ is not in  $\upsilon^+(\mathbb I_{\gamma})$. Hence no summand of  $Z$ is a pull-back of a cycle by $\beta: \mathbb I \to S$. Let $H_3(\mathbb I, \mathbb Z) \cong \bigoplus_{a+b = 3}H_a(S, \mathbb Z) \otimes H_b(\mathbb P^1, \mathbb Z)$ be the K\"unneth isomorphism, then no summand of $Z$ defines a class in $H_1(S, \mathbb Z) \otimes H_2(\mathbb P^1, \mathbb Z)$. This implies that the class of $Z$ is zero and hence the theorem follows.   To complete the proof we show the above claim. Let $t \in S$, in what follows $W_t \subset W$ is the pencil of quadrics defined by $t$, that is, the image of $\upsilon^+(\mathbb I_t)$ by $f: \tilde W \to W$. Now let $t \in S - \Gamma$, then, by our assumption (c), $\Gamma$ is in $\delta(S) - \Sing \delta(S)$, hence $W_t$ is not in the family $\lbrace  W_{\ell}, \  \ell \in \Gamma \rbrace$. Since $W_t$, $W_{\ell}$ are lines they intersect in at most one point and  the same is true for $\upsilon^+(\mathbb I_t)$, $\upsilon^+(\mathbb I_{\ell})$. Let $W_{\gamma} = \bigcup_{\ell \in \Gamma} W_{\ell}$, this implies by (a) that the intersection of $\mathbb I_t$ and ${\upsilon^{+}}^{-1}(W_{\gamma})$ has real dimension $\leq 1$. Hence $\upsilon^+( \mathbb I_t)$ is not in  $\upsilon^+(\mathbb I_{\gamma})$.
 \end{proof}
Since $H^3(\tilde W', \mathbb Z)$  admits nonzero torsion and this is a birational invariant then $\tilde W$ is not rational. Since it is unirational, it is a counterexample to L\"uroth problem, see \cite{AM}.
\par So far, we have performed our goal of explicitly reconstructing the $2$-torsion cohomology of an Artin-Mumford double solid $\tilde W$ from that of its Enriques surface $S$. The $2$-torsion group
$H_1(S, \mathbb Z)$ embeds in $H^3(\tilde W', \mathbb Z)$ via the map $\mathsf c: H_1(S, \mathbb Z) \to \mathbb H_3(\tilde W', \mathbb Z)$ and Poincar\'e duality. \par  This offers a partially
new proof of Artin-Mumford counterexample to L\"uroth problem, based on some geometry of Enriques surfaces, more precisely of Reye congruences of lines.
    
 \ \par 
 {\tiny }
\end{document}